\documentclass[11pt,a4paper]{article}

\usepackage{amsfonts,amsgen,amstext,amsbsy,amsopn,amssymb,amscd}
\usepackage[leqno]{amsmath}
\usepackage[amsmath,amsthm,thmmarks]{ntheorem}
\usepackage{mathrsfs}
\usepackage{graphicx}

\newtheorem{thm}{Theorem}[section]
\newtheorem{prop}[thm]{Proposition}

\newtheorem{lem}[thm]{Lemma}
\newtheorem{example}[thm]{Example}
\newtheorem{cor}[thm]{Corollary}
\newtheorem{fact}[thm]{Fact}
\newtheorem{conj}[thm]{Conjecture}

\theoremstyle{definition}
\newtheorem{defn}[thm]{Definition}
\newtheorem{claim}[thm]{Claim}

\makeatletter
\@addtoreset{equation}{section}
\makeatother

\def\hu{\mathcal{U}}

\def\hl{\mathcal{L}}
\def\hht{\mathcal{T}}

\def\hf{\mathcal{F}}

\def\hk{\mathcal{K}}
\def\ha{\mathcal{A}}
\def\hb{\mathcal{B}}
\def\hd{\mathcal{D}}
\def\hs{\mathcal{S}}
\def\hr{\mathcal{R}}
\def\hc{\mathcal{C}}

\title{\bf\Large Stability for
 Helly-type and triangle-free families}
\author{Peter Frankl$^1$, Jian Wang$^2$\\[10pt]
$^{1}$R\'{e}nyi Institute, Budapest, Hungary\\[6pt]
$^{2}$School of Mathematics, Sichuan University, Chengdu, 610065, China.\\[6pt]
E-mail:  $^1$frankl.peter@renyi.hu, $^2$wangjianmath01@scu.edu.cn
}
\date{}

\begin{document}

\maketitle

\begin{abstract}
We consider $k$-graphs, $\hf\subset \binom{[n]}{k}$, $k\geq 3$. A $k$-graph is called intersecting if any two of its edges have non-empty intersection. It is called a star if all its edges share a common vertex. The $k$-graph $\hf$ is called Helly if all its intersecting subfamilies  are stars. If the same is required only for subfamilies consisting of three edges, it is called triangle-free. It is well known that for $n\geq 3k/2$, the full star is the unique largest triangle-free family whence the largest Helly family as well. In 1984 Tuza proved the best possible bound $|\hf|\leq \binom{n-k-1}{k-1}+\binom{n-2}{k-2}+1$ for Helly families that are not stars, albeit only for some unspecified $n>n_0(k)$.

The aim of this paper is twofold. First we establish the same bound for $n>2k$. Second we show that for $n>12k^2$ the same upper bound holds for triangle-free families. It is shown as well that it is not true for $2k<n\leq 3k-4$.
\end{abstract}

\section{Introduction}

Let $[n]=\{1,2,\ldots,n\}$ be the standard $n$-set and let $\binom{[n]}k$ denote the collection of all $k$-subsets of $[n]$. A subset $\hf$ of $\binom{[n]}{k}$ is called a {\it $k$-uniform} family or simply a {\it $k$-graph}.

A $k$-graph $\hf$ is called {\it intersecting} if $F\cap F'\neq \emptyset$ for all $F,F'\in \hf$. Let us state the classical Erd\H{o}s-Ko-Rado Theorem, one of the cornerstones of extremal set theory.

\begin{thm}[\cite{EKR}]\label{thm-ekr}
Let $n\geq 2k$ and suppose that $\hf\subset \binom{[n]}{k}$ is intersecting. Then
\begin{align}\label{ineq-ekr}
|\hf| \leq \binom{n-1}{k-1}.
\end{align}
\end{thm}

No doubt motivated
by \eqref{ineq-ekr} in 1971 Erd\H{o}s \cite{E} stated the following analogous problem.

First a definition. Three sets $F_1,F_2,F_3$ are said to form a {\it triangle} if $F_i\cap F_j\neq \emptyset$ for $1\leq i<j\leq 3$ but $F_1\cap F_2\cap F_3=\emptyset$.

A $k$-graph is called {\it triangle-free} if it contains no three edges forming a triangle.

\begin{conj}[Erd\H{o}s \cite{E}]
Let $n\geq \frac{3}{2}k$, $k\geq 3$. If $\hf\subset \binom{[n]}{k}$ is triangle-free then
\begin{align}\label{ineq-triangle-free}
|\hf| \leq \binom{n-1}{k-1}.
\end{align}
\end{conj}

 Note that for $2k>n\geq \frac{3}{2}k$, $\binom{[n]}{k}$ is intersecting. Thus $\hf\subset \binom{[n]}{k}$ is triangle-free if and only if $F_1\cap F_2\cap F_3\neq \emptyset$ for all $F_1,F_2,F_3\in \hf$. Under this assumption \eqref{ineq-ekr} was proved in \cite{F76}.

 Later the first author proved Erd\H{o}s' conjecture for $n\geq n_0(k)$ as well. Eventually Mubayi and Verstra\"{e}te \cite{MV} extended it to the full range.

  \begin{thm}[\cite{F76}, \cite{F81}, \cite{MV}]\label{thm-F76}
 Suppose that $n\geq \frac{3}{2}k$, $k\geq 3$ and $\hf\subset \binom{[n]}{k}$ is triangle-free. Then \eqref{ineq-triangle-free} holds.
 \end{thm}

 Bollob\'{a}s and Duchet \cite{BD} proved \eqref{ineq-triangle-free} under a considerably stronger restriction.

 \begin{defn}
 A $k$-graph $\hf$ is said to be of {\it Helly-type} if for all intersecting subfamilies $\hf_0\subset \hf$ the {\it total intersection} $\cap_{F\in \hf_0} F$ is non-empty.
 \end{defn}

 \begin{thm}[\cite{BD}]
 Let $n\geq 2k\geq 6$. Suppose that $\hf\subset \binom{[n]}{k}$ is of Helly-type. Then
 \begin{align}\label{ineq-1.3}
 |\hf| \leq \binom{n-1}{k-1}.
 \end{align}
 \end{thm}

 The proof of Bollob\'{a}s and Duchet is based on a simple observation. For a $k$-graph $\hf$, an edge $F\in \hf$ and a $(k-1)$-subset $G\subset F$, we say that $G$ is {\it unique} if $G\not\subset F'$ holds for all $F\neq F'\in \hf$.

 \begin{fact}
 Suppose that $k\geq 3$, $\hf\subset \binom{[n]}{k}$ is of Helly-type. Then every edge of $\hf$ has at least one unique $(k-1)$-subset.
 \end{fact}
 \begin{proof}
 Suppose the contrary and fix $F\in \hf$ such that for each of its $(k-1)$-sets, $G_1,G_2,\ldots,G_k$, there exits $F_i\in \hf$, $F_i\neq F$ with $G_i\subset F_i$. Note that $k\geq 3$ implies that $\{F,F_1,\ldots, F_k\}$ is intersecting. However $F\cap F_1\cap \ldots\cap F_k=\emptyset$, a contradiction.
 \end{proof}

 For the Erd\H{o}s-Ko-Rado Theorem, nearly 60 years ago Hilton and Milner established a strong stability result. To state it we need a definition. A family $\hf\subset \binom{[n]}{k}$ is called {\it non-trivial} if $\cap \hf=\emptyset$.

 \begin{thm}[\cite{HM67}]\label{thm-HM}
 Let $n> 2k$.  Suppose that $\hf\subset \binom{[n]}{k}$ is intersecting and non-trivial. Then
\begin{align}\label{ineq-nontrival}
|\hf| \leq \binom{n-1}{k-1}- \binom{n-k-1}{k-1} +1.
\end{align}
\end{thm}

The aim of the present paper is to establish stability results for both Helly-type and triangle-free $k$-graphs.

\begin{example}[\cite{Tuza94}]
Let $T_0=\binom{[n]}{k}$ satisy $T_0\cap \{1,2\}=\{2\}$. Define $\hht(n,k):= \{T_0\}\cup \hht_1\cup \hht_2$, where
\[
\hht_1=\left\{T\in \binom{[n]}{k}\colon 1\in T,T\cap T_0=\emptyset\right\},\ \hht_2=\left\{T\in \binom{[n]}{k}\colon \{1,2\}\subset T\right\}.
\]
\end{example}

It is easy to verify that $\hht(n,k)$ has the Helly property. Namely 1 is common to all edges in $\hht(n,k)\setminus \{T_0\}$ and 2 is common to all edges intersecting $T_0$.

\begin{thm}\label{thm-main}
Suppose that $\hf\subset \binom{[n]}{k}$ is non-trivial and has Helly property, $n\geq 2k\geq 10$. Then
\begin{align}\label{ineq-1.4}
|\hf| \leq |\hht(n,k)| =1+\binom{n-k-1}{k-1}+\binom{n-2}{k-2}.
\end{align}
\end{thm}

Let us mention that the above bound $n\geq 2k$ is best possible. Namely, for $n<2k$, $\hf$ is automatically intersecting. Hence the Helly property implies that it is a star.

\begin{thm}\label{thm-main1}
Suppose that $\hf\subset \binom{[n]}{k}$ is non-trivial and triangle-free, $n\geq 12k^2$ and $k\geq 7$. Then
\begin{align}\label{ineq-1.5}
|\hf| \leq |\hht(n,k)| =1+\binom{n-k-1}{k-1}+\binom{n-2}{k-2}.
\end{align}
\end{thm}

Since triangle-free is a weaker restriction, \eqref{ineq-1.4} follows from \eqref{ineq-1.5} whenever \eqref{ineq-1.5} is valid. However  \eqref{ineq-1.5} does not hold for $n\leq 3k-4$.

Define
\[
\ha(n,k) :=
 \left\{
 F\in\binom{[n]}k\colon
 |F\cap [4]|\geq 3
 \right\}.
\]
It is easy to check that $\ha(n,k)$ is non-trivial and triangle-free.

\begin{prop}\label{prop-counterexample-near-three}
For $k\geq 7$ and $2k\leq n\leq 3k-4$,
\[
 |\ha(n,k)|>|\hht(n,k)|.
\]
\end{prop}

The proof of \eqref{ineq-1.4} is based on the two lemmas and induction. The base case of the induction is $n=2k$. For this we rely on a recent result of Cambie and Salia \cite{CambieSalia}.

\begin{prop}[\cite{CambieSalia}]\label{prop-main}
Suppose that $\hf\subset \binom{[2k]}{k}$ is non-trivial and of Helly-type, $k\geq 3$. Then
\[
|\hf| \leq  |\hht(2k,k)|=\binom{2k-2}{k-2}+2.
\]
\end{prop}

To state the lemmas let us recall some notation. For a family $\hf\subset 2^{[n]}$ and a set $P$, define
\[
\hf(P):=\left\{F\setminus P\colon P\subset F\in \hf\right\},\ \hf(\bar{P}):=\left\{F\in \hf\colon F\cap P=\emptyset\right\}.
\]
For $P=\{x\}$, we use the shorthand $\hf(x)$ and $\hf(\bar{x})$. Let us also define
\[
\hf(x,\bar{P})=\left\{F\setminus \{x\}\colon F\in \hf, F\cap (P\cup \{x\})=\{x\}\right\}.
\]
For $P=\{y\}$, we use the shorthand $\hf(x,\bar{y})$.

\begin{lem}\label{lem-key1}
Let $k\ge5$, $n\ge 2k+1$, and let
$\hf\subseteq\binom{[n]}k$ be a non-trivial Helly family.  If
$|\hf|>|\hht(n,k)|$, then there exists some  $x\in [n]$ such that
\begin{equation}
 |\hf(x)|\le \binom{n-k-2}{k-2}+\binom{n-3}{k-3}=|\hht(n,k)|-|\hht(n-1,k)|.
\end{equation}
\end{lem}

For $\hf\subset \binom{[n]}{k}$, a set $T\subset [n]$ is called {\it a transversal} of $\hf$ if $T\cap F\neq \emptyset$ for any $F\in \hf$.

\begin{lem}\label{lem-key2}
Let $k\ge4$, $n\ge2k+1$, and let $\hf\subset \binom{[n]}{k}$ be a non-trivial Helly family.  If there exist $x,c\in [n]$ such that $\{x,c\}$ is a transversal of $\hf$  and $|\hf(x)|\le  \binom{n-k-2}{k-2}+\binom{n-3}{k-3}$, then
\[
|\hf|\le |\hht(n,k)|.
\]
\end{lem}

\section{Proofs of  Proposition \ref{prop-counterexample-near-three} and Theorem \ref{thm-main} and preliminaries}

\begin{proof}[Proof of Proposition \ref{prop-counterexample-near-three}]
Note that
\[
 |\mathcal A(n,k)|= 4\binom{n-4}{k-3}+\binom{n-4}{k-4}.
\]
By applying Vandermonde's identity, we obtain
\begin{align*}
 \binom{n-2}{k-2}
 &=
 \binom{n-4}{k-2}
 +2\binom{n-4}{k-3}
 +\binom{n-4}{k-4}.\end{align*}
It follows that
\begin{align*}
 |\mathcal A(n,k)|-\binom{n-2}{k-2}
 &=
 2\binom{n-4}{k-3}
 -\binom{n-4}{k-2}= \frac{3k-n-3}{k-2}\binom{n-4}{k-3}.
 \label{eq-A-minus-second-layer}
\end{align*}
We are left to show that
\begin{align}\label{ineq-1.6}
 \frac{3k-n-3}{k-2}\binom{n-4}{k-3}>\binom{n-k-1}{k-1}.
\end{align}
Let $f(n)=\frac{3k-n-3}{k-2}\binom{n-4}{k-3}/\binom{n-k-1}{k-1} $. Then for $2k\leq n\leq 3k-6$ and $k\geq 7$,
\[
\frac{f(n+1)}{f(n)} = \frac{3k-n-4}{3k-n-3} \cdot \frac{n-3}{n-k}\cdot \frac{n-2k+1}{n-2k}>1.
\]
It follows that
\[
f(3k-5)\geq f(3k-6)\geq \ldots\geq f(2k) =\frac{k-3}{k-2}\binom{2k-4}{k-3}>1.
\]
We also need to show $f(3k-4)>1$.
Let $g(k)=f(3k-4)=\frac{1}{k-2}\binom{3k-8}{k-3}/\binom{2k-5}{k-1}$. Then for $k\geq 5$,
\[
\frac{g(k+1)}{g(k)} = \frac{3k(k-3)(3k-5)(3k-7)}
{4(k-1)(k-2)(2k-3)^2}>1.
\]
Since $g(4),g(5)>1$, \eqref{ineq-1.6} follows.
\end{proof}

\begin{proof}[Proof of Theorem \ref{thm-main}]
We prove the statement by induction on $n$.
Note that the base case $n=2k$  follows from Proposition  \ref{prop-main}.  Now let $n\ge 2k+1$ and assume the result for $n-1$.

Suppose that \eqref{ineq-1.4} is false.
Lemma~\ref{lem-key1} provides a vertex $x$ with
$|\hf(x)|\leq \binom{n-k-2}{k-2}+\binom{n-3}{k-3}$.  The subfamily $\hf(\bar{x})$
is nonempty, since otherwise $\hf$ would be a star of center $x$.  It is
also of Helly-type, being a subfamily of $\hf$.

If $\hf(\bar{x})$ is non-trivial, then the induction hypothesis gives
\[
 |\hf|
 =|\hf(\bar{x})|+|\hf(x)|
 \le |\hht(n-1,k)|+\binom{n-k-2}{k-2}+\binom{n-3}{k-3}=|\hht(n,k)|,
\]
a contradiction.  If $\hf(\bar{x})$ is trivial, assume that $\hf(\bar{x})$ is a star of center $c$, then $\{x,c\}$ is a transversal of $\hf$. Then
$|\hf|\le |\hht(n,k)|$ follows from Lemma~\ref{lem-key2}.
\end{proof}

For $x\in [n]$ and $\hf\subset \binom{[n]}{k}$ define a partition of $\hf(x)$:
\[
\hs_x =\{G\in \hf(x)\colon |\hf(G)|= 1\},\ \hl_x =\{G\in \hf(x)\colon |\hf(G)|\geq 2\}.
\]
Set $\hs =\cup_{x\in [n]} \hs_x$ and note  $\hs\cap \hl_x=\emptyset$.

Counting pairs $(G,F)$, $G\in \hf(x)$ in two ways yields
\[
k|\hf| =\sum_{x\in [n]} |\hl_x| +|\hs|.
\]
Define $M=\max\{|\hl_x|\colon x\in [n]\}$. Using $|\hl_x|+|\hs| \leq \binom{n}{k-1}$ we infer
\begin{align}\label{ineq-2.1}
M\geq \frac{k|\hf|-\binom{n}{k-1}}{n-1}.
\end{align}

The following observation is from \cite{F81}.

\begin{fact}[\cite{F81}]\label{fact-2.1}
Suppose that $G_1,G_2,G_3\in \hf(x)$ for some $x\in [n]$. If $G_1\cap G_3=\emptyset$ and $G_1\cap G_2\neq \emptyset\neq G_2\cap G_3$ then $G_2\in \hs_x$.
\end{fact}

Define a graph $G_x$ with vertex set $\hl_x$ and edge set
\[
\{(E_1,E_2)\colon E_1,E_2\in \hl_x,\ E_1\cap E_2\neq \emptyset\}.
\]
Let $\hk\subset \hl_x$ form a connected component of $G_x$ and let $V(\hk) =\cup_{K\in \hk}K$. Define
\[
\hk^+=\hk \cup \{E\in \hs_x\colon E\cap V(\hk)\neq \emptyset\}.
\]

The next lemma is from \cite{MV}. For self-containedness we include the simple proof.

\begin{lem}[\cite{MV}]\label{lem-MV}
Let $\hf\subset \binom{[n]}{k}$ be a triangle-free family and let $x\in [n]$. If $\hk\subset \hl_x$ forms a connected component of $G_x$, then $\hk^+$ is an intersecting family.
\end{lem}
\begin{proof}
Suppose for contradiction that $A,B\in \hk^+$ but $A\cap B=\emptyset$. Then there is a shortest path $A_0=A,A_1,A_2,\ldots,A_t=B$ with $A_1,A_2,\ldots,A_{t-1}\in \hk$ connecting $A,B$. Note that $t\geq 2$ and $A_0\cap A_2=\emptyset$. Now $A_0\cap A_1\neq \emptyset \neq A_1\cap A_2$ and Fact  \ref{fact-2.1} imply $A_1\in \hs_x$, contradicting $A_1\in \hk\subset \hl_x$.
\end{proof}

\begin{lem}
Let $v_1,v_2,\ldots,v_s$ be positive integers with $\sum_{1\leq i\leq s}v_i \leq n-1$.
Set
\[
t=\min\left\{k+2,\left\lfloor\frac{n-1}{2}\right\rfloor\right\}.
\]
If $v_i\leq n-1-t$, $i=1,2,\ldots,s$, then
\begin{align}\label{ineq-2.3}
\sum_{1\leq i\leq s} \binom{v_i-1}{k-2} \leq\binom{n-t-2}{k-2}+\binom{t-1}{k-2}.
\end{align}
\end{lem}

\begin{proof}
If $k-2\leq b\leq a$, then
\begin{align}\label{ineq-2.2}
\binom{a}{k-2}+\binom{b}{k-2}\leq \binom{a+1}{k-2}+\binom{b-1}{k-2},
\end{align}
which follows from $\binom{b-1}{k-3}\leq \binom{a}{k-3}$. Without loss of generality, assume that $v_1\geq v_2\geq \ldots\geq v_s$. Since $v_1-1\leq n-t-2$, apply \eqref{ineq-2.2} with $(a,b)=(v_1-1,v_i-1)$ for some $i\in \{2,3,\ldots,s\}$ repeatedly we obtain that
\[
\sum_{1\leq i\leq s} \binom{v_i-1}{k-2} \leq\binom{n-t-2}{k-2}+ \sum_{2\leq i\leq s} \binom{v_i'-1}{k-2},\ \sum_{2\leq i\leq s}v_i' \leq t.
\]
By applying \eqref{ineq-2.2} repeatedly again,
\[
\sum_{2\leq i\leq s} \binom{v_i'-1}{k-2} \leq  \binom{\sum_{2\leq i\leq s} v_i' -s+1}{k-2} \leq \binom{t-1}{k-2}.
\]
Thus \eqref{ineq-2.3} follows.
\end{proof}

For convenience let us introduce the notation:
\[
\hf[x]=\{F\in \hf\colon x\in F\}, \ \hf[x,\bar{B}] =\{F\in \hf\colon x\in F,\ F\cap B=\emptyset\}
\]
and $\hf[x,\bar{y}] =\hf[x,\overline{\{y\}}]$.

Note that by definition for every $B\in \hf(\bar{y})$ the family $\{B\}\cup (\hf[y]\setminus \hf[y,\bar{B}])$ is intersecting. By the Helly-property, we can fix an element $z_B\in B$ that is contained in each $F\in \hf[y]\setminus \hf[y,\bar{B}]$. We call $z_B$ the {\it local center of $B$} with respect to $y$.

For any $x$, $B\in \binom{[n]\setminus \{x\}}{k}$ and for a fixed $z\in B$, define
\[
\hht_{x,B}^{(1)}:=\left\{E\in \binom{[n]\setminus \{x\}}{k-1}\colon z\in E\right\},\ \hht_{x,B}^{(2)}:=\left\{E\in \binom{[n]\setminus \{x\}}{k-1}\colon E\cap B= \emptyset\right\}.
\]
Let $\hht_{x,B}=\hht_{x,B}^{(1)}\cup \hht_{x,B}^{(2)}$. It is easy to see that $\hht_{x,B}$ is isomorphic to $\hht(n,k)(1)$. Moreover, $\hf(x)\subset \hht_{x,B}$ holds for any $x$ and $B\in \binom{[n]\setminus \{x\}}{k}$.

Let $\hd=\left\{D\in \hf[\bar{y}, z_B]\colon z_D\in B,\ D\neq B\right\}$.

\begin{claim}\label{claim-center-payment}
If $\hf(x,\bar{B})\neq \emptyset$, then $|\hht_{x,B}^{(2)}\setminus \hf(x)| \geq |\hd|$.
\end{claim}

\begin{proof}
 Let $U=\cup_{D\in \hd} (D\setminus B)$ and set $u=|U|$.   We claim that every $E\in \hf(x,\bar{B})$ satisfies $E\cap U=\emptyset$. Indeed, if $E\cap U\neq \emptyset$ then $E\cap D\neq \emptyset$ for some $D\in\hd$. However, $E\cap B=\emptyset$ and $z_D\in B$ imply $z_D\notin E$, a contradiction.

Fix some $E\in \hf(x,\bar{B})$.  Since every $D\in\hd$ contains
$z_B$ and $D\neq B$, we infer that
\[
 |\hd|\leq \binom{u+k-1}{k-1}-1.
\]
On the other hand, every member of $\binom{U\cup E}{k-1}\setminus\{E\}$ intersects $U$ and is disjoint from $B$. Thus none of them is in $\hf(x)$.  Moreover, the number of these $(k-1)$-sets is exactly $\binom{u+k-1}{k-1}-1$. Thus,
\[
|\hht_{x,B}^{(2)}\setminus \hf(x)| \geq  \binom{u+k-1}{k-1}-1 \geq  |\hd|.
\]
\end{proof}

The {\it matching number $\nu(\hf)$} of $\hf$ is defined as the maximum number of pairwise disjoint members in $\hf$.
The {\it covering number $\tau(\hf)$} of $\hf$ is defined as the minimum size of $T$ such that $T$ is a transversal, that is, $T\cap F \neq \emptyset$ for all $F\in \hf$.

We need the following inequality generalizing the Erd\H{o}s-Ko-Rado Theorem.

\begin{prop}[\cite{F87}]
Suppose that $\hf\subset \binom{[n]}{k}$ and $n\geq (\nu(\hf)+1)k$. Then
\begin{align}\label{ineq-EMCup}
|\hf|\leq \nu(\hf)\binom{n-1}{k-1}.
\end{align}
\end{prop}

\begin{prop}[\cite{FW22}]
Let $n,k,i$ be positive integers. Then
\begin{align}\label{ineq-key0}
\binom{n-i}{k} \geq \frac{n-ik}{n}\binom{n}{k}, \mbox{\rm \ for } n>ik.
\end{align}
\end{prop}

Let us prove four more simple statements.

\begin{fact}\label{fact-6.1}
Suppose that $\hf$ is triangle-free and it contains
the sunflower $\{C\cup D_1,\ldots,C\cup D_s\}$,  $C\neq \emptyset$, $D_1,D_2,\ldots,D_s$ are pairwise disjoint. Then for every $E\in \hf(\bar{C})$, $E\cap D_i\neq \emptyset$ can hold for at most one value of $1\leq i\leq s$.
\end{fact}

\begin{proof}
Indeed, if $E\cap D_i\neq \emptyset \neq E\cap D_j$ then $E,C\cup D_i,C\cup D_j$ form a triangle, a contradiction.
\end{proof}

\begin{fact}\label{fact-6.2}
 Let $\hf\subset \binom{[n]}{k}$ be triangle-free. Assume that  $\nu(\hf(C))\geq s$ and set $q=n-|C|-(s-1)(k-|C|)$, then
\[
|\hf(\bar{C})|\leq s\binom{q}{k-1}.
\]
\end{fact}

\begin{proof}
Let $D_1,D_2,\ldots,D_s$ be a matching in $\hf(C)$ and let $U=[n]\setminus (C\cup D_1\cup D_2\cup \ldots\cup D_s)$. For $E\in \hf(\bar{C})$, by Fact \ref{fact-6.1} $E\subset U\cup D_i$ for some $1\leq i\leq s$. Since $\hf\cap \binom{U\cup D_i}{k}$ is triangle-free, by Theorem \ref{thm-F76} we infer $|\hf\cap \binom{U\cup D_i}{k}|\leq \binom{q}{k-1}$ for $|U\cup D_i|>\frac{3k}{2}$ and $|\hf\cap \binom{U\cup D_i}{k}|\leq \binom{q}{k}<\binom{q}{k-1}$ for $|U\cup D_i|\leq \frac{3k}{2}$.  Thus,
\[
|\hf(\bar{C})|\leq \sum_{1\leq i\leq s} \left|\hf\cap \binom{U\cup D_i}{k}\right|\leq  s\binom{q}{k-1}.
\]
\end{proof}

\begin{fact}\label{fact-6.3}
Fix $x_1,x_2\in [n]$ and let $\hf\subset \binom{[n]}{k}$ be triangle-free. Then either $\hf(x_1,\overline{x_2})$ and $\hf(\overline{x_1},x_2)$ are vertex-disjoint or
\[
|\hf(x_1,x_2)| \leq \binom{n-2}{k-2}-\binom{n-k-1}{k-2} \leq (k-1) \binom{n-3}{k-3}.
\]
\end{fact}

\begin{proof}
Arguing indirectly we may fix $F_1,F_2\in \hf$ with $F_i\cap \{x_1,x_2\}=\{x_i\}$, $i=1,2$,   $1\leq |F_1\cap F_2|\leq k-1$. Note that $F_1\cap F_2\cap\{x_1,x_2\}=\emptyset$ and the triangle-free property imply that  $F_1\cap F_2$ is a transversal of $\hf(x_1,x_2)$. Hence,
\[
|\hf(x_1,x_2)| \leq \binom{n-2}{k-2} -\binom{n-2-|F_1\cap F_2|}{k-2} \leq \binom{n-2}{k-2}
 -\binom{n-k-1}{k-2}.
 \]
\end{proof}

\begin{lem}\label{lem-5.1}
Let $[n]=X_1\cup X_2\cup \ldots\cup X_s$ be a partition of $[n]$ into non-empty subsets. If $\hf\subset \binom{[n]}{k}$ is a family satisfying $\hf=\cup_{1\leq i\leq s} \hf\cap \binom{X_i}{k}$ and $\hf\cap \binom{X_i}{k}$ is an intersecting family for each $i=1,2,\ldots,s$, then for $k\geq 3$
\begin{align}\label{ineq-5.2}
|\hf|\leq  \binom{n-s+1}{k-1}.
\end{align}
\end{lem}

\begin{proof}
Let $\hf_i= \hf\cap \binom{X_i}{k}$.
Note that if $|X_i|\geq 2k$ then by the Erd\H{o}s-Ko-Rado theorem $|\hf_i|\leq \binom{|X_i|-1}{k-1}\leq \binom{|X_i|}{k-1}$. If $|X_i|< 2k$ then $|\hf_i|\leq \binom{|X_i|}{k}\leq \binom{|X_i|}{k-1}$. Thus,
\[
|\hf|\leq \sum_{1\leq i\leq s} \binom{|X_i|}{k-1}.
\]
Let us consider an $a$-set $A$ and a $b$-set $B$, $a,b\geq k-1$ and $|A\cap B|=k-2$. Then $\binom{A}{k-1}\cap \binom{B}{k-1}=\emptyset$ implies
\begin{align}\label{ineq-easy}
\binom{a}{k-1}+\binom{b}{k-1}=\left|\binom{A}{k-1}\right|+\left|\binom{B}{k-1}\right|&\leq \left|\binom{A\cup B}{k-1}\right|\nonumber\\[3pt]
&=\binom{a+b-(k-2)}{k-1}\leq \binom{a+b-1}{k-1}.
\end{align}
If $1\leq a<k-1$ or  $1\leq b<k-1$,   then $\binom{a}{k-1}+\binom{b}{k-1}\leq \binom{a+b-1}{k-1}$ holds as well. Using \eqref{ineq-easy}, \eqref{ineq-5.2} follows.
\end{proof}

\section{Proof of Lemma \ref{lem-key1}}

We need the following inequalities, whose proofs are deferred to Section 6.

\begin{lem}\label{lem-3.1}
For $k\ge5$ and $n\ge2k+1$, set
\[
t=t(n,k):=\min\left\{k+2,\left\lfloor\frac{n-1}{2}\right\rfloor\right\}.
\]
Then the following equalities hold:
\begin{align}
 k|\hht(n,k)|-\binom n{k-1}
 &>(n-1)\left(
 \binom{n-t-2}{k-2}+\binom{t-1}{k-2}
 \right),                                                   \label{eq:pi-concentration}\\
 \binom{n-k-4}{k-3}&\ge \binom{t-1}{k-1},                                 \label{eq:pi-rigidity}\\
 \binom{n-3}{k-2}-\binom{n-k-2}{k-2}
 &>n+\binom{t-1}{k-1}.                                                     \label{eq:pi-global-gap}
\end{align}
\end{lem}

\begin{lem}\label{lem-4.2}
Let $k\ge5$, $n\ge 2k+1$, and let
$\hf\subseteq\binom{[n]}k$ be a non-trivial Helly family.  If
$|\hf|>|\hht(n,k)|$, then there exist some  $x,c\in [n]$ such that
\begin{align}
|\hf(x,\bar c)| \leq  \binom{t-1}{k-1}.
 \label{eq-h-bound}
\end{align}
where $t=\min\{k+2,\lfloor (n-1)/2\rfloor\}$.
\end{lem}

\begin{proof}
Let $M:=\max_{x\in[n]}|\hl_x|$
and choose $x$ with $|\hl_x|=M$.
Since $|\hf|>|\hht(n,k)|$, by \eqref{ineq-2.1} and \eqref{eq:pi-concentration},
\begin{equation}
 |\hl_x|>\frac{k|\hht(n,k)|-\binom n{k-1}}{n-1}>\binom{n-t-2}{k-2}+\binom{t-1}{k-2}.
 \label{eq-M-lower}
\end{equation}

Let $v_1,\ldots,v_s$ be the sizes of the supports of the components of of $G_x$.  By Lemma~\ref{lem-MV} and  the Helly property, we infer that
\[
 |\hl_x|\leq\sum_{i=1}^s\binom{v_i-1}{k-2},
 \qquad
 \sum_{i=1}^s v_i\leq n-1.
\]
If $v_i\leq n-1-t$ for all $i$, then
\[
 |\hl_x|\leq\binom{n-t-2}{k-2}+\binom{t-1}{k-2},
\]
which contradicts \eqref{eq-M-lower}.  Thus some
component $\hk$ of $G_x$ has support $V=V(\hk)$ satisfying
\begin{equation}
 |([n]\setminus\{x\})\setminus V|\leq t-1.
 \label{eq-large-support}
\end{equation}

By Lemma~\ref{lem-MV} and the Helly property, $\hk^+$ is a star. Let $c$ be the center of this star.  Then every $E\in \hf(x)$ with $c\notin E$ satisfies $E\cap V=\emptyset$. Then \eqref{eq-h-bound} follows from
 \eqref{eq-large-support}.
\end{proof}

\begin{proof}[Proof of Lemma \ref{lem-key1}]
By Lemma \ref{lem-4.2}, there exist $x,c\in \hf$ such that
\[
|\hf(x,\bar c)| \leq  \binom{t-1}{k-1}.
\]
Assume for contradiction that
\begin{align}\label{ineq-assumption}
 |\hf(x)|>
 \binom{n-k-2}{k-2}+\binom{n-3}{k-3}.
\end{align}

We distinguish two cases.

{\bf Case 1. }$\hf(\bar x,\bar c)\neq \emptyset$.

Let $B\in \hf(\bar x,\bar c)$ and let $z_B$ be the local center of $B$ with respect to $c$.
Hence,
\[
 \hf(x,c)\subseteq\hu_B:=
 \binom{[n]\setminus (B\cup \{x,c\})}{k-2}
 \bigcup
 \left\{Q\in\tbinom{[n]\setminus \{x,c\}}{k-2}\colon z_B\in Q\right\}.
\]
The two parts of $\hu_B$ are disjoint and
\[
 |\hu_B|=\binom{n-k-2}{k-2}+\binom{n-3}{k-3}.
\]
Since by our assumption \eqref{ineq-assumption} we have
$|\hf(x)|=| \hf(x,c)|+| \hf(x,\bar{c})|>|\hu_B|$, it follows that
\begin{equation}
 |\hu_B\setminus\hf(x,c)|= |\hu_B|-|\hf(x,c)|<|\hf(x,\bar{c})|\leq\binom{t-1}{k-1}.
 \label{eq-rh}
\end{equation}

For any $F\in\hf(\bar x,\bar c)\setminus \{B\}$, let $z_F$ be the local center of $F$ with respect to $c$.
Consequently,
\[
 \hr(F):=\{Q\in\hu_B\colon Q\cap F\neq\emptyset,\ z_F\notin Q\}
 \subseteq\hu_B\setminus\hf(x,c).
\]

\begin{claim}
$\hf(\bar x,\bar c)$ forms a sunflower with center $B\setminus \{z_B\}$. That is, there is some $Z\subset [n]\setminus (B\cup \{x,c\})$ such that
\begin{equation}
 \hf(\bar x,\bar c)
 =\{B\}\cup
 \{F_z=(B\setminus\{z_B\})\cup\{z\}:z\in Z\}
 \label{eq-clones}
\end{equation}
Moreover, the local center of $F_z$ with respect to $c$ is $z$.
\end{claim}

\begin{proof}
Let  $F\in \hf(\bar{x},\bar{c})\setminus \{B\}$. The following cases exhaust
all possibilities.

\begin{enumerate}
\item If $z_F=z_B$, then there is some $y\in F\setminus B$. The number of sets in $\tbinom{[n]\setminus (B\cup \{x,c\})}{k-2}$ containing $y$ is $\binom{n-k-3}{k-3}$, which implies $|\hr(F)|\geq \binom{n-k-3}{k-3}$.

\item If $z_F\in B\setminus\{z_B\}$ and $z_B\in F$, then the number of $(k-2)$-sets $E\subset [n]\setminus \{x,c\}$ with $E\cap F=\{z_B\}$ is $\binom{n-k-2}{k-3}$, which implies $|\hr(F)|\geq \binom{n-k-2}{k-3}$.   If $z_F\in B\setminus\{z_B\}$ and
$z_B\notin F$, then let $y\in F\setminus B$ and the number of sets in $\tbinom{[n]\setminus (B\cup \{x,c\})}{k-2}$ containing $y$ is $\binom{n-k-3}{k-3}$, which implies $|\hr(F)|\geq \binom{n-k-3}{k-3}$.

\item If $z_F\in [n]\setminus (B\cup \{x,c\})$ and $z_B\in F$, then the number of $(k-2)$-sets $E\subset [n]\setminus \{x,c\}$ with $E\cap F=\{z_B\}$ is $\binom{n-k-2}{k-3}$, which implies $|\hr(F)|\geq \binom{n-k-3}{k-3}$.

\item Suppose that $z_F\in[n]\setminus(B\cup\{x,c\})$
 and $z_B\notin F$.
Consider
\[
 \left\{Q\in
 \binom{[n]\setminus(B\cup\{x,c,z_F\})}{k-2}\colon  Q\cap\bigl(F\setminus (B\cup \{z\})\bigr)
 \neq\emptyset \right\}.
\]
There are $\binom{n-k-3}{k-2}-
 \binom{n-k-2-|F\setminus B|}{k-2}$
such sets.  Indeed, the first term counts all $(k-2)$-sets disjoint
from $B\cup \{x,c,z_F\}$, while the second term counts those
which are also disjoint from
$F\setminus (B\cup\{z_F\})$.

If $|F\setminus B|\geq2$, choose $w\in F\setminus (B\cup\{z_F\})$, then every set
\[
 Q\in
 \binom{[n]\setminus(B\cup\{x,c,z_F\})}{k-2}
 \quad\text{with}\quad w\in Q
\]
belongs to $\hr(F)$. Consequently,
\[
 |\hr(F)|\geq\binom{n-k-4}{k-3}.
\]

It remains to consider $|F\setminus B|=1$.  Since $z_F\in F\setminus B$,
we then have $F\setminus B=\{z_F\}$.
Because $|F|=|B|=k$, the set $F$ is obtained from $B$ by deleting
exactly one element and adding $z_F$.  Since $z_B\in B$ but $z_B\notin F$, the
deleted element must be $z_B$.  Hence,
\[
 F=(B\setminus\{z_B\})\cup\{z_F\}.
\]
\end{enumerate}

Thus unless $F=(B\setminus\{z_B\})\cup\{z_F\}$ holds for
 every $F\in \hf(\bar{x},\bar{c})\setminus \{B\}$, by \eqref{eq:pi-rigidity} we have
\[
|\hu_B\setminus\hf(x,c)|\geq |\hr(F)| \geq\binom{n-k-4}{k-3}
 \geq\binom{t-1}{k-1},
\]
contradicting
\eqref{eq-rh}. Thus,
\begin{equation}
 \hf(\bar x,\bar c)
 =\{B\}\cup
 \{F_z=(B\setminus\{z_B\})\cup\{z\}:z\in Z\}
 \label{eq-clones}
\end{equation}
for some $Z\subseteq[n]\setminus(B\cup\{x,c\})$. Moreover, the local center
of each $F_z$ is $z$.
\end{proof}

For $z\in Z$, there are
$\binom{n-k-3}{k-3}$ $(k-2)$-sets in $[n]\setminus \{x,c\}$ that contain $z$ and are disjoint from $B$.  Since $|\hu_B\setminus \hf(x,c)|<\binom{t-1}{k-1}<\binom{n-k-3}{k-3}$, in view of \eqref{eq-rh} and \eqref{eq:pi-rigidity}, at least one of these $\binom{n-k-3}{k-3}$ sets belongs to
$\hf(x,c)$.  That is, for each $F_z\in \hf(\bar{x},\bar{c})\setminus \{B\}$ there exists $E_z\in \hf(x,c)$ such that $E_z\cap F_z=\{z\}$.
Since $z$ is the local center of $F_z$ with respect to $c$, we infer that
\begin{equation}
 E\in\hf(c),\quad E\cap F_z\neq\emptyset
 \quad\Longrightarrow\quad z\in E.
 \label{eq-global-clone-center}
\end{equation}
Similarly, among the
$\binom{n-k-2}{k-3}>\binom{t-1}{k-1}$ $(k-2)$-sets $E$ in  $[n]\setminus \{x,c\}$ with $E\cap B=\{z_B\}$, at least one belongs to $\hf(x,c)$. Note that $\hf(c)\subseteq \hht_{c,B}$
and $|\hht_{c,B}|=|\hht(n,k)|-1$.
Let $R_c:=|\hht_{c,B}\setminus\hf(c)|$. Then
\begin{align*}
|\hf|&=|\hf (c)|+|\hf(x,\bar{c})|+|\hf(\bar{x},\bar{c})|\\[3pt]
&= (|\hht_{c,B}|-R_c) +|\hf(x,\bar{c})|+(|Z|+1) \\[3pt]
&=|\hht(n,k)| +|\hf(x,\bar{c})|+|Z|-R_c.
\end{align*}
By $|\hf|>|\hht(n,k)|$,
\begin{align}\label{ineq-leftcase}
|\hf(x,\bar{c})|+|Z|> R_c.
\end{align}

 If $Z\neq\emptyset$, fix
$z\in Z$.  By \eqref{eq-global-clone-center}, every member of
\[
\left\{E\in \binom{[n]\setminus \{c\}}{k-1}\colon z_B\in E,\  z\notin E,\ E\cap (B\setminus\{z_B\})\neq \emptyset\right\}
\] is
in  $\hht_{c,B}^{(1)}\setminus \hf(c)$.  There are exactly $\binom{n-3}{k-2}-\binom{n-k-2}{k-2}$
such members.  Hence, by
\eqref{eq:pi-global-gap}, \eqref{eq-h-bound}, and
$|Z|\leq n-k-2<n$,
\[
 R_c\geq
 \binom{n-3}{k-2}-\binom{n-k-2}{k-2}
 >n+\binom{t-1}{k-1}>|Z|+|\hf(x,\bar c)|,
\]
contradicting \eqref{ineq-leftcase}.

It remains to consider the case $Z=\emptyset$.  Then by \eqref{ineq-leftcase} we have $|\hf(x,\bar{c})|> R_c$.  Consider the sets in $\hf[x,\bar c]$.  For each  $F\in \hf[x,\bar c]$, let $z_F$ be the local center of $F$ with respect to $c$.
If $z_F\neq z_B$, then all $(k-1)$-sets $E$ with $E\cap \{z_B,z_F\}=\{z_B\}$ and $E\cap F\neq \emptyset$ are in $\hht_{c,B}^{(1)}\setminus \hf(c)$.  Their number
is at least $\binom{n-3}{k-2}-\binom{n-k-2}{k-2}$.
Thus,
\[
 R_c\geq\binom{n-3}{k-2}-\binom{n-k-2}{k-2}>|\hf[x,\bar{c}]|,
\]
contradicting \eqref{ineq-leftcase}.

We may therefore assume that all sets $F$ in $\hf[x,\bar{c}]$ satisfy $z_F=z_B$.  If $\hf(c,\bar{B})\neq \emptyset$, then by applying
Claim~\ref{claim-center-payment} with $x=c$ and $\hd=\hf[x,\bar{c}]$ (noting that $B\in \hf(\bar{x},\bar{c})$ implies $B\notin\hf[x,\bar{c}]$), we obtain that
$R_c\geq |\hht_{c,B}^{(2)}\setminus\hf(c)| \geq |\hf[x,\bar{c}]|$.  If all sets in $\hf(c)$ intersect $B$, then $\hht_{c,B}^{(2)} \subset \hht_{c,B}\setminus \hf(c)$. It follows that
\[
R_c\geq \binom{n-k-1}{k-1}>\binom{t-1}{k-1}\geq |\hf[x, \bar{c}]|,
\]
contradicting \eqref{ineq-leftcase} again.

{\bf Case 2. } $\hf(\bar x,\bar c)=\emptyset$.

Since $\hf$ is non-trivial, $\hf[x,\bar{c}]\neq \emptyset$. Choose $B\in\hf[x,\bar c]$ and let $z_B$ be the local center of $B$ with respect to $c$.  Define $R_c:=|\hht_{c,B}\setminus\hf(c)|$. Then
\begin{align*}
|\hf|&=|\hf (c)|+|\hf(x,\bar{c})|= (|\hht_{c,B}|-R_c) +|\hf(x,\bar{c})| =|\hht(n,k)|-1 +|\hf(x,\bar{c})|-R_c.
\end{align*}
By $|\hf|>|\hht(n,k)|$,
\begin{align}\label{ineq-leftcase2}
|\hf(x,\bar{c})|\geq  R_c.
\end{align}
For any $F\in \hf[x,\bar{c}]$, let $z_F$ be the local center of $F$ with respect to $c$.
If there is some $F\in \hf[x,\bar{c}]$ such that $z_F\neq z_B$, then  all $(k-1)$-sets $E\subset [n]\setminus \{c\}$ with $E\cap \{z_B,z_F\}=\{z_B\}$ and $E\cap F\neq \emptyset$ are in $\hht_{x,B}^{(1)}\setminus \hf(c)$.  It follows that
\[
 R_c\geq\binom{n-3}{k-2}-\binom{n-k-2}{k-2}
 \overset{\eqref{eq:pi-global-gap}}{>}\binom{t-1}{k-1}\geq |\hf[x,\bar{c}]|,
\]
contradicting \eqref{ineq-leftcase2}.
Otherwise $z_F=z_B$ for all $F\in \hf[x,\bar{c}]$. If $\hf(c,\bar{B})\neq \emptyset$, then by
applying Claim~\ref{claim-center-payment} with $x=c$ and $\hd =\hf[x,\bar{c}]\setminus \{B\}$ we obtain that
$R_c\geq |\hf[x,\bar{c}]|-1$. If all sets in $\hf(c)$ intersect $B$, then
\[
R_c\geq |\hht_{c,B}^{(2)}|=\binom{n-k-1}{k-1}>|\hf[x,\bar{c}]|,
 \]
contradicting \eqref{ineq-leftcase2}.  Thus,
\[
 |\hf(x)|\leq
 \binom{n-k-2}{k-2}+\binom{n-3}{k-3}.
\]
\end{proof}

\section{Proof of Lemma \ref{lem-key2}}

We need the following inequality.

\begin{lem}\label{lem-4.1}
For $k\ge 4$ and $n\ge2k+1$,
\begin{equation}
 \binom{n-3}{k-2}-\binom{n-k-2}{k-2}\ge
 \binom{n-3}{k-3}+\binom{n-k-3}{k-3}.
 \label{eq:pi-bridge}
\end{equation}
\end{lem}

\begin{proof}
Note that
\[
\binom{n-3}{k-2}-\binom{n-3}{k-3} = \frac{n-2k+2}{n-k}\binom{n-3}{k-2}
\]
and
\[
\binom{n-k-2}{k-2}+\binom{n-k-3}{k-3}=\frac{n-4}{n-k-2}\binom{n-k-2}{k-2}.
\]
Then \eqref{eq:pi-bridge} is equivalent to
\[
\frac{\binom{n-3}{k-2}}{\binom{n-k-2}{k-2}} \geq \frac{(n-4)(n-k)}{(n-k-2)(n-2k+2)}=1+\frac{2 (k-2) (n - k-1)}{(n-k-2)(n-2k+2)},
\]
Since
\[
\frac{\binom{n-3}{k-2}}{\binom{n-k-2}{k-2}}= \prod_{i=0}^{k-2}\left(1+\frac{k-1}{n-k-2-i}\right)
 \ge1+\frac{(k-2)(k-1)}{n-k-2},
\]
it suffices to show
\[
k-1>\frac{2 (n - k-1)}{n-2k+2},
\]
which is true for $n\geq 2k+1$ and $k\geq 4$. This proves \eqref{eq:pi-bridge}.
\end{proof}

\begin{proof}[Proof of Lemma \ref{lem-key2}]
Since $\hf$ is non-trivial, both $\hf[x,\bar c]$ and $\hf[c,\bar x]$ are nonempty.
Moreover, there are disjoint members
\[
 B\in\hf[x,\bar c]
 \qquad\hbox{and}\qquad
 C\in\hf[c,\bar x].
\]
Indeed, if $\hf[x,\bar c]$ and $\hf[c,\bar x]$ are cross-intersecting, then $\hf$ is intersecting. Thus by the Helly property $\hf$  must be a star,  contradicting the non-triviality of $\hf$.

 Let $z_B\in B$ be the local center of $B$ with respect to $c$. Recall that $\hf(c)\subset \hht_{c,B}$ and $|\hht_{c,B}|=|\hht(n,k)|-1$.
 Since  $\{x,c\}$ is a
transversal,
\begin{equation}
 |\hf(\bar{c})|=|\hf[x,\bar c]|\leq|\hf(x)|\leq
 \binom{n-k-2}{k-2}+\binom{n-3}{k-3}.
 \label{eq-explicit-degree}
\end{equation}
Thus,
\[
 |\hf|=|\hf[c]|+|\hf(\bar{c})|\leq |\hht(n,k)|-1
 -|\hht_{c,B}\setminus\hf(c)|+|\hf[x,\bar c]|.
\]
Consequently, it is  enough to prove
\begin{equation}
 |\hht_{c,B}\setminus\hf(c)|
 \geq |\hf[x,\bar c]|-1.
 \label{eq-explicit-missing-target}
\end{equation}

For each $F\in\hf[x,\bar c]$, let $z_F\in F$ be the local center of $F$ with respect to $c$. Then
\begin{equation}
 E\in\hf(c),\quad E\cap F\neq\emptyset
 \quad\Longrightarrow\quad z_F\in E.
 \label{eq-local-center-two-cover}
\end{equation}
Define
\[
 \hc_{\mathrm{in}}:=\left\{
 F\in\hf[x,\bar c]\setminus\{B\}:z_F\in B
 \right\},\qquad \hc_{\mathrm{ext}}:=
 \left\{
 F\in\hf[x,\bar c]\setminus\{B\}:z_F\notin B
 \right\}.
\]
Note that $C\in \hf[c,\bar{B}]$.  Apply
 Claim~\ref{claim-center-payment} with $x=c$ and
$\hd= \hc_{\mathrm{in}}$, we obtain that
\begin{equation}
 |\hht_{c,B}^{(2)}\setminus \hf(c)|
 \geq
 \left| \hc_{\mathrm{in}}\right|.
 \label{eq-B-center-payment}
\end{equation}

If there is no $F\in \hf[x,\bar c]$ satisfying $z_F\notin B$,  then $| \hc_{\mathrm{in}}|=|\hf[x,\bar{c}]|-1$. Then  \eqref{eq-explicit-missing-target} follows from \eqref{eq-B-center-payment}.
 Thus we may assume that $ \hc_{\mathrm{ext}}\neq \emptyset$.

 \begin{claim}\label{claim-1}
 Unless \eqref{eq-explicit-missing-target} already holds,  there exists some $Z\subset [n]\setminus (B\cup \{c\})$ such that
 \[
  \hc_{\mathrm{ext}}= \{F_z=(B\setminus\{z_B\})\cup\{z\}\colon z \in Z\}.
\]
Moreover, $z$ is the local center of $F_z$ with respect to $c$.
 \end{claim}
 \begin{proof}
 If there is $F\in  \hc_{\mathrm{ext}}$ and $|F\setminus B|\geq2$,
choose
$w\in F\setminus (B\cup \{z_F\})$.  By \eqref{eq-local-center-two-cover}, all sets in \[
\left\{E\in \binom{[n]\setminus \{c\}}{k-1}\colon E\cap B=\emptyset,\ w\in E, z_F\notin E\right\}
\]
 are in $\hht_{c,B}^{(2)}\setminus \hf(c)$, their  number is $\binom{n-k-3}{k-2}$.
Moreover,  all sets in
\[
\left\{E\in \binom{[n]\setminus \{c\}}{k-1}\colon z_B\in E,\ E\cap F\neq \emptyset,\ z_F\notin E\right\}
 \]
are in $\hht_{c,B}^{(1)}\setminus \hf(c)$.  If $z_B\notin F$, then their number is exactly $\binom{n-3}{k-2}-\binom{n-k-2}{k-2}$,
and if $z_B\in F$ it is at least this large.  Thus, by \eqref{eq:pi-bridge} and \eqref{eq-explicit-degree},
\begin{align*}
 |\hht_{c,B}\setminus\hf(c)|& = |\hht_{c,B}^{(2)}\setminus\hf(c)|+ |\hht_{c,B}^{(1)}\setminus\hf(c)|\\
 &\geq \binom{n-k-3}{k-2}
 +\binom{n-3}{k-2}-\binom{n-k-2}{k-2}\\
 &\geq\binom{n-k-2}{k-2}+\binom{n-3}{k-3}\\
 &\geq|\hf[x,\bar c]|
\end{align*}
and \eqref{eq-explicit-missing-target} holds.

Consequently we may  assume that $|F\setminus B|=1$ for every $F\in  \hc_{\mathrm{ext}}$.  Since $z_F\notin B$ and $|F\setminus B|=1$ , we have
\[
 F=(B\setminus\{b\})\cup\{z_F\}
\]
for some $b\in B$.  If $b\neq z_B$, then $z_B\in F$ and every
set in $\{E\in \binom{[n]\setminus \{c\}}{k-1}\colon z_B\in E,\ z_F\notin E\}$ is in $\hht_{c,B}^{(1)}\setminus \hf(c)$.  Then  by \eqref{eq:pi-bridge}
\[
 |\hht_{c,B}\setminus\hf(c)|\geq\binom{n-3}{k-2}> \binom{n-3}{k-2}-\binom{n-k-2}{k-2}
 \geq\binom{n-k-2}{k-2}+\binom{n-3}{k-3}
 \geq|\hf[x,\bar c]|
\]
and \eqref{eq-explicit-missing-target} holds again.  Thus for every $F\in \hc_{ext}$,
\[
 F=(B\setminus\{z_B\})\cup\{z_F\}
\]
 and the claim holds.
\end{proof}

By Claim \ref{claim-1},
\begin{align*}
 |\hc_{\mathrm{ext}}|\leq n-k-1<n.
\end{align*}
Fix
$F_0\in\hc_{\mathrm{ext}}$ and write
\[
 F_0=(B\setminus\{z_B\})\cup\{z_0\}.
\]
Every member of
\[
 \left\{
 E\in\binom{[n]\setminus \{c\}}{k-1}:
 z_B\in E,\
 z_0\notin E,\
 E\cap(B\setminus\{z_B\})\neq\emptyset
 \right\}
\]
is in  $\hht_{c,B}^{(1)}\setminus \hf(c)$.  Thus,
\begin{align}\label{ineq-4.2}
|\hht_{c,B}^{(1)}\setminus \hf(c)|\geq  \binom{n-3}{k-2}-\binom{n-k-2}{k-2}\overset{\eqref{eq:pi-global-gap}}{>}n>|\hc_{ext}|.
\end{align}
Adding \eqref{ineq-4.2} and \eqref{eq-B-center-payment}, we obtain that
\begin{align*}
|\hht_{c,B}\setminus \hf(c)|=|\hht_{c,B}^{(1)}\setminus \hf(c)|+|\hht_{c,B}^{(2)}\setminus \hf(c)| \geq &|\hc_{\mathrm{in}}|+
 |\hc_{\mathrm{ext}}|
 =|\hf[x,\bar c]|-1.
\end{align*}
This completes the proof.
\end{proof}

\section{Proof of Theorem \ref{thm-main1}}

Let us first prove Theorem \ref{thm-main1} in the case $|\hf(\bar{y})|=1$  for some $y$.

\begin{lem}\label{lem-6.1}
Let $y\in [n]$, $P\in \hf(\bar{y})$. Define \[
\hf_P(y)=\{F\setminus \{y\}\colon y\in F\in \hf,\ F\cap P\neq \emptyset\}.
\]
Then $\hf_P(y)$ is intersecting.
\end{lem}

\begin{proof}
In the opposite case we can fix $R_1,R_2\in \hf(y)$ with $R_1\cap R_2\cap P=\emptyset$, $R_1\cap P\neq \emptyset \neq R_2\cap P$. Consequently $R_1\cup \{y\}, R_2\cup \{y\}$ and $P$ form a triangle, a contradiction.
\end{proof}

\begin{cor}\label{cor-6.1}
If $|\hf(\bar{y})|=1$ and $n\geq 2k$ then
\[
|\hf| \leq |\hht| =\binom{n-k-1}{k-1}+\binom{n-2}{k-2}+1.
\]
Moreover, equality holds if and only if $\hf$ is isomorphic to $\hht$.
\end{cor}

\begin{proof}
Let $\hf(\bar{y})=\{P\}$. By Lemma \ref{lem-6.1}, $\hf_P(y)$ is intersecting.
By Theorem \ref{thm-ekr}, $|\hf_P(y)|\leq \binom{n-2}{k-2}$ with equality if and only if it is the full star of some $x\in P$. As $\hf(\{y\},P\cup \{y\})\subset \binom{[n]\setminus (\{y\}\cup P)}{k-1}$, $|\hf(\{y\},P\cup \{y\})|\leq \binom{n-k-1}{k-1}$. Thus the corollary follows.
\end{proof}

We need the following inequalities, whose proofs are deferred to Section 6.

\begin{lem}\label{lem-5.2}
Let $k\ge 5$ and $n\ge 12k^2$.
Then the following equalities hold:
\begin{align}
&\binom{n-k-5}{k-2}
>\binom{n-2}{k-2}-\binom{n-k-1}{k-2}+1,                                 \label{ineq-6.1}\\[3pt]
&\binom{n-k-2}{k-2}\ge
 \left(\frac{n-2}{k-2}-k-3\right)\binom{k^2+3k-10}{k-1},\label{ineq-6.2}\\[3pt]
&\left(\frac{n-1}{k-1}-k-1\right)\binom{k^2+2k-4}{k-1}
 <\binom{n-4}{k-3},\label{ineq-6.3}\\[3pt]
&\binom{n-4}{k-3}<\binom{n-k-2}{k-2}.\label{ineq-6.4}
\end{align}
\end{lem}

\begin{lem}\label{lem-5.3}
Let $\hf\subset \binom{[n]}{k}$ be a triangle-free family with $|\hf|>|\hht(n,k)|$. If $n\geq 12 k^2$ and $k\geq 7$, then there exist $x,y\in [n]$ such that
\begin{align}\label{ineq-5.4}
|\hf(x,y)|\geq  \binom{n-k-5}{k-2}.
\end{align}
\end{lem}

\begin{proof}
Let $M:=\max_{x\in[n]}|\hl_x|$
and choose $x$ with $|\hl_x|=M$.
Since $|\hf|>|\hht(n,k)|$, by \eqref{ineq-2.1} and \eqref{eq:pi-concentration},
\begin{equation}
 |\hl_x|>\frac{k|\hht(n,k)|-\binom n{k-1}}{n-1}>\binom{n-k-4}{k-2}+\binom{k+1}{k-2}.
 \label{eq-M-lower2}
\end{equation}

Let $v_1,\ldots,v_s$ be the sizes of the supports of the components of of $G_x$. By symmetry assume that
\[
v_1\geq v_2\geq \ldots\geq v_s\geq k-1.
\]
If $v_1\leq 2k$, then $s\geq \lceil\frac{n-1}{2k}\rceil> k+4$.
By Lemma  \ref{lem-5.1},
\[
\binom{n-k-4}{k-2}+\binom{k+1}{k-2}<|\hl_x| \leq  \binom{(n-1)-s+1}{k-2}\leq \binom{n-k-4}{k-2},
\]
a contradiction. Thus $v_1\geq 2k+1$.
By Lemma~\ref{lem-MV} and  the Erd\H{o}s-Ko-Rado Theorem, we infer that
\[
 |\hl_x|\leq \binom{v_1-1}{k-2}+\sum_{2\leq i\leq s}\binom{v_i}{k-2},
 \qquad
 \sum_{i=1}^s v_i\leq n-1.
\]
If $v_1\leq n-k-5$, then by applying $\binom{a}{k-2}+\binom{b}{k-2}\leq \binom{a+1}{k-2}+\binom{b-1}{k-2}$ for $k-2\leq b\leq a$ repeatedly, we obtain that
\[
 \binom{n-k-4}{k-2}+\binom{k+1}{k-2}<|\hl_x|\leq \binom{n-k-5}{k-2}+\binom{k+3}{k-2}.
\]
It implies that
\[
\binom{n-k-5}{k-3}<\binom{k+1}{k-3}+\binom{k+2}{k-3}<2\binom{k+2}{k-3},
\]
which leads to a contradiction as $n\geq 12k^2$ and $k\geq 7$. Thus $v_1\geq n-k-4$. Let $\hk$ be
 the largest
component of $G_x$ with support $V=\cup_{K\in \hk}K$. Then
\begin{equation}
 |([n]\setminus\{x\})\setminus V|=n-1-v_1\leq k+3.
 \label{eq-large-support2}
\end{equation}
It follows that for $n\geq 12k^2$ and $k\geq 7$
\[
|\hk| \geq \binom{n-k-4}{k-2}-\binom{k+3}{k-1}\geq|\hl_x|-\binom{k+3}{k-1}\geq \binom{n-k-5}{k-2}.
\]

Since $\binom{n-k-5}{k-2}>\binom{n-2}{k-2}-\binom{n-k-1}{k-2}+1$ from \eqref{ineq-6.1}, by Theorem \ref{thm-HM} we infer that  $\hk$ is a star, say of center $y$.
Then
\[
|\hf(x,y)|\geq |\hk|\geq  \binom{n-k-5}{k-2}.
\]
\end{proof}

\begin{lem}\label{lem-5.4}
Let $\hf\subset \binom{[n]}{k}$ be a triangle-free family with $|\hf|>|\hht(n,k)|$. If $n\geq 12 k^2$ and $k\geq 7$, then there exists $y\in [n]$ such that
\[
|\hf(y)|\geq \binom{n-k-2}{k-1}.
\]
\end{lem}

\begin{proof}
By Lemma \ref{lem-5.3}, there exist $x,y\in [n]$ such that $|\hf(x,y)|\geq  \binom{n-k-5}{k-2}$.  Let $s=\lfloor\frac{n-2}{k-2}\rfloor-k-3$.
Note that
\begin{align}\label{ineq-5.3}
|\hf(x,y)|\geq  \binom{n-k-5}{k-2}&\overset{\eqref{ineq-key0}}{\geq} \frac{n-2-(k+3)(k-2)}{n-2} \frac{n-2}{k-2}\binom{n-3}{k-3}\geq s\binom{n-3}{k-3}.
\end{align}
By \eqref{ineq-EMCup}, we infer that $\hf(x,y)$ has matching number at least $s$. Applying Fact \ref{fact-6.2} with $C=\{x,y\}$ and
\[
q=n-2-(s-1)(k-2) < k^2+3k-10,
\]
 we obtain that
\begin{align*}
|\hf(\bar{x},\bar{y})|< s\binom{q}{k-1}<s\binom{k^2+3k-10}{k-1}.
\end{align*}

Now by \eqref{ineq-6.2},
\begin{align*}
|\hf(x,y)|+|\hf(\bar{x},\bar{y})|&\leq \binom{n-2}{k-2}+\left(\frac{n-2}{k-2}-k-3\right)\binom{k^2+3k-10}{k-1}\\[3pt]
&\leq \binom{n-2}{k-2}+\binom{n-k-2}{k-2}.
\end{align*}
Then
\[
|\hf(x,\bar{y})|+|\hf(\bar{x},y)| \geq |\hf|-\binom{n-2}{k-2}-\binom{n-k-2}{k-2} > \binom{n-k-2}{k-1}+1.
\]

Since by \eqref{ineq-5.3} and $n\geq 12 k^2$ we have
\[
|\hf(x,y)|\geq s\binom{n-3}{k-3}> (k-1) \binom{n-3}{k-3},
\]
 in view of Fact \ref{fact-6.3}, $\hf(x,\bar{y})$ and $\hf(\bar{x},y)$ are vertex-disjoint.  Let $|\cup \hf(x,\bar{y})|=m_1$ and $|\cup \hf(\bar{x},y)|=m_2$. Clearly $m_1+m_2\leq n-2$.  By Lemma  \ref{lem-5.1},
\[
\binom{n-k-2}{k-1}+1\leq |\hf(x,\bar{y})|+|\hf(\bar{x},y)|\leq \binom{m_1}{k-1} +\binom{m_2}{k-1}.
\]
Then we claim that either $m_1\leq  k$ or $m_2\leq k$. Indeed, otherwise we have
\[
\binom{n-k-2}{k-1}+1\leq  \binom{m_1}{k-1} +\binom{m_2}{k-1}\leq  \binom{n-k-3}{k-1} +\binom{k+1}{k-1}.
\]
Equivalently,
\[
\binom{n-k-3}{k-2}+1\leq  \binom{k+1}{k-1},
\]
a contradiction.

By symmetry assume $m_1\leq k$. Then $ |\hf(x,\bar{y})|\leq \binom{m_1}{k-1}\leq k$ and
\[
|\hf(\bar{x},y)|\geq \binom{n-k-2}{k-1}+1-|\hf(x,\bar{y})|\geq \binom{n-k-2}{k-1}+1-k.
\]
Thus,
\[
|\hf(y)|=|\hf(x,y)|+|\hf(\bar{x},y)| \geq  \binom{n-k-5}{k-2} +\binom{n-k-2}{k-1}+1- k\geq \binom{n-k-2}{k-1}.
\]
\end{proof}

\begin{proof}[Proof of Theorem \ref{thm-main1}]
Let $s=\lfloor\frac{n-1}{k-1}\rfloor-k-1$.
By Lemma \ref{lem-5.4}  there exists $y\in [n]$ such that
\[
|\hf(y)|\geq \binom{n-k-2}{k-1} \overset{\eqref{ineq-key0}}{\geq}  \frac{n-1-(k+1)(k-1)}{n-1} \frac{n-1}{k-1}\binom{n-2}{k-2} \geq s\binom{n-2}{k-2}.
\]
 By \eqref{ineq-EMCup}, we infer that $\hf(y)$ has matching number at least $s$. Apply Fact \ref{fact-6.2} with $C=\{y\}$ and
\[
q=  n-1-(s-1)(k-1)\leq k^2+2k-4.
 \]
By \eqref{ineq-6.3} we obtain that
\[
|\hf(\bar{y})| \leq s \binom{k^2+2k-4}{k-1}\leq \left(\frac{n-1}{k-1} -k-1\right)\binom{k^2+2k-4}{k-1}<\binom{n-4}{k-3}.
\]

For  $P\in \hf(\bar{y})$ and define
\[
 \ha_P:=\{E\in\hf(y)\colon E\cap P\ne\emptyset\},\qquad
 \hb_P:=\{E\in\hf(y)\colon E\cap P=\emptyset\}.
\]
Clearly $|\hb_P| \leq \binom{n-k-1}{k-1}$.

\begin{claim}
We may assume that $\ha_P$ is a star with center in $P$ for any $P\in \hf(\bar{y})$.
\end{claim}
\begin{proof}
By Lemma \ref{lem-6.1}, $\ha_P$ is intersecting. If $\ha_P$ is non-trival intersecting, then
\[
|\ha_P| \leq \binom{n-2}{k-2}-\binom{n-k-1}{k-2}+1\overset{\eqref{ineq-6.4}}{<} \binom{n-2}{k-2}-\binom{n-4}{k-3}+1.
\]
Hence,
\[
|\hf| \leq |\ha_P|+|\hb_P|+|\hf(\bar{y})|<\binom{n-k-1}{k-1}+\binom{n-2}{k-2}+1
\]
and  we are done.

If $\ha_P$ is a star with center in $[n]\setminus (P\cup \{y\})$, then
\[
|\ha_P| \leq \binom{n-2}{k-2}-\binom{n-k-2}{k-2}\overset{\eqref{ineq-6.4}}{<} \binom{n-2}{k-2}-\binom{n-4}{k-3}.
\]
It follows that
\[
|\hf| \leq |\ha_P|+|\hb_P|+|\hf(\bar{y})|<\binom{n-k-1}{k-1}+\binom{n-2}{k-2}+1
\]
and we are done again. Thus the claim holds.
\end{proof}

Let $z_P\in P$ be a center of $\ha_P$ for any $P\in \hf(\bar{y})$. By Corollary \ref{cor-6.1} we may assume $|\hf(\bar{y})|\geq 2$. Let  $P,Q\in \hf(\bar{y})$. Note that
\[
\hf(y) \subset \hs_P:=\left\{E\subset \binom{[n]\setminus \{y\}}{k-1}\colon \mbox{ either }z_P\in E \mbox{ or } E\cap P=\emptyset \right\}
\]
and
\[
|\hs_P|=\binom{n-2}{k-2}+\binom{n-k-1}{k-1}.
\]
Then $\hf(y)\subset \hs_P\cap \hs_Q$. If $|\hs_P\setminus \hs_Q|\geq \binom{n-4}{k-3}$, then
\[
|\hf| \leq |\hs_P\cap \hs_Q|+|\hf(\bar{y})|< |\hs_P|-|\hs_P\setminus \hs_Q|+\binom{n-4}{k-3}\leq \binom{n-2}{k-2}+\binom{n-k-1}{k-1}.
\]
Thus we may assume $|\hs_P\setminus \hs_Q|< \binom{n-4}{k-3}$ and $|\hs_Q\setminus \hs_P|< \binom{n-4}{k-3}$.

If $z_P\notin Q$, let $z\in Q\setminus \{z_Q\}$, then  all the $(k-1)$-sets $E\subset [n]\setminus \{y\}$  with $E\cap \{z_P,z_Q,z\}=\{z_P,z\}$ are in $\hs_P\setminus\hs_Q$. It follows that
$|\hs_P\setminus\hs_Q|\geq \binom{n-4}{k-3}$, a contradiction. By symmetry, we may assume $z_P,z_Q\in P\cap Q$.

Since $P\neq Q$, by symmetry we may assume $z\in P\setminus Q$. Then all $(k-1)$-sets $E\subset [n]\setminus \{y\}$ with $z\in E$ and $E\cap Q=\emptyset$ are all in $\hs_Q\setminus\hs_P$. It follows that
$|\hs_Q\setminus\hs_P|\geq \binom{n-k-2}{k-2}\overset{\eqref{ineq-6.4}}{>}  \binom{n-4}{k-3}$, a contradiction again.
\end{proof}

\section{Proof of Lemma \ref{lem-3.1} and Lemma \ref{lem-6.1}}

\begin{proof}[Proof of Lemma \ref{lem-3.1}]
Note that
$$t=\left\{
                \begin{array}{ll}
                  k+2, & n\geq 2k+5; \\[3pt]
                  k+1, & n=2k+3, 2k+4; \\[3pt]
                  k, & n=2k+1, 2k+2.
                \end{array}
              \right.
$$
Let us first prove \eqref{eq:pi-concentration}  for $n\geq 2k+5$. Let
$m=n-k-4\geq k+1$. Note that
\begin{align*}
\binom{n-2}{k-2}&=\sum_{0\leq j\leq k-2} \binom{k+2}{k-2-j}\binom{m}{j},\\[3pt]
\binom{n-k-1}{k-1}&= \binom{m}{k-1}+3\binom{m}{k-2}+3\binom{m}{k-3}+\binom{m}{k-4}\\[3pt]
\binom{n}{k-1}&=\sum_{0\leq j\leq k-1} \binom{k+4}{k-1-j}\binom{m}{j},\\[3pt]
(n-1)\binom{n-k-4}{k-2} &=(k-1)\binom{m}{k-1}+(2k+1)\binom{m}{k-2}\\[3pt]
(n-1)\binom{k+1}{k-2} &= \binom{k+1}{3} \binom{m}{1}+(k+3)\binom{k+1}{3}.
\end{align*}
Then the left side
of \eqref{eq:pi-concentration} minus
the right side has a Vandermonde expansion
\begin{equation}
 \sum_{j=0}^{k-1}a_j\binom mj.
 \label{eq:pi-vandermonde}
\end{equation}
For $2\le j\le k-3$, writing $s=k-1-j$, one obtains
\begin{equation}
 a_j=k\binom3s+k\binom{k+2}{s-1}-\binom{k+4}s>0,
 \label{eq:pi-coeff-middle}
\end{equation}
which follows from
\[
 \frac{\binom{k+4}s}{\binom{k+2}{s-1}}
 =\frac{(k+4)(k+3)}{s(k+4-s)}<k.
\]
The two leading coefficients are
\[
 a_{k-1}=0,\qquad a_{k-2}=k-5\geq 0.
\]
The remaining coefficients are
\begin{align*}
 a_1&=k\binom{k+2}{5}-\binom{k+4}{6}-\binom{k+1}{3}\nonumber\\
 &=\frac{k(k-1)(k+1)}{720}
 (5k^3-9k^2-50k-144)>0 \mbox{ for }k\geq 5,\\
 a_0&=k\binom{k+2}{4}-\binom{k+4}{5}
 -(k+3)\binom{k+1}{3}+k\nonumber\\
 &=\frac{k}{30}(k^4-5k^3-25k^2-10k+39)>0\mbox{ for }k\ge9.
\end{align*}
Since $m=n-k-4\geq k+1$, one can also check that
\[
a_0+a_1 m \geq  a_0+a_1 (k+1)>0 \mbox{ holds for }k=6,7,8.
\]
For $k=5$, \eqref{eq:pi-concentration} is equivalent to
\[
 7 n^2 - 127 n + 470>0,
\]
which is true for $n\geq 13$.  This proves \eqref{eq:pi-concentration} for $n\ge2k+5$.

For $n=2k+s$ with $s=1,2,3,4$,  $t=k-1+\lceil s/2\rceil$.
The left hand of \eqref{eq:pi-concentration} equals
\begin{align*}
&k\left(\binom{2k+s-2}{k-2}+\binom{k+s-1}{s}+1\right)
-\binom{2k+s}{k-1}\\[3pt]
&\qquad = \left(k-\frac{(2k+s)(2k+s-1)}{(k-1)(k+s+1)}\right)\binom{2k+s-2}{k-2}+k\binom{k+s-1}{s}+k.
\end{align*}
Let $\rho_s=\frac{(2k+s)(2k+s-1)}{(k-1)(k+s+1)}$.
Then \eqref{eq:pi-concentration} is equivalent to
\begin{align}\label{ineq-4.3}
(k-\rho_s)\binom{2k+s-2}{k-2}&+k\binom{k+s-1}{s}+k\nonumber\\[3pt]
-(2k+s-1)&\left(\binom{k-1+\lfloor s/2\rfloor}{\lfloor s/2\rfloor+1}
+\binom{k-2+\lceil s/2\rceil}{\lceil s/2\rceil}\right) >0.
\end{align}
For $k=5,6,7,8$ and $s=1,2,3,4$, one can check directly that \eqref{ineq-4.3} holds. Thus we may assume $k\geq 9$.
Let
\[
P_s:= k\binom{k+s-1}{s}+k
-(2k+s-1)\left(\binom{k-1+\lfloor s/2\rfloor}{\lfloor s/2\rfloor+1}
+\binom{k-2+\lceil s/2\rceil}{\lceil s/2\rceil}\right).
\]
By simplifying,
\begin{align*}
 P_1&=-3k^2+5k,\\
 P_2&=\frac{-k^3-2k^2+5k+2}{2},\\
 P_3&=\frac{k(k^3-9k^2+2k+18)}6>0, \mbox{ for }k\geq 9,\\
 P_4&=\frac{k(k^4-2k^3-25k^2+2k+72)}{24}>0\mbox{ for }k\geq 5.
\end{align*}
For $k\ge 9$, $\rho_1,\rho_2\le 4$. It implies that
\[
(k-\rho_s)\binom{2k+s-2}{k-2}\geq (k-4)\binom{2k+s-2}{k-2}\geq  (k-4)\binom{2k+s-2}{2} \mbox{ for }s=1,2.
\]
Since
\begin{align*}
(k-4)\binom{2k+s-2}{2}+P_s >0\mbox{ for }k\geq 9,
\end{align*}
we conclude that \eqref{eq:pi-concentration} holds for all $n\ge2k+1$.

Next we prove \eqref{eq:pi-rigidity}. If $n\geq 2k+5$, then $n-k-4\geq k+1$. Hence
\[
\binom{n-k-4}{k-3}\geq \binom{k+1}{k-3}=\binom{k+1}{4}\ge \binom{k+1}{2} \mbox{ for }k\geq 5.
\]
For $n=2k+3$ or $2k+4$ then $t=\lfloor \frac{n-1}{2}\rfloor=k+1$. Now
\[
\binom{(2k+4)-k-4}{k-3} >\binom{(2k+3)-k-4}{k-3}=\binom{k-1}{2}>k=\binom{t-1}{k-1} \mbox{ for }k\geq 5.
\]
For $n=2k+1$ or $2k+2$, then $t=\lfloor \frac{n-1}{2}\rfloor=k$. Now
\[
\binom{(2k+2)-k-4}{k-3} >\binom{(2k+1)-k-4}{k-3}=\binom{k-3}{k-3}=1=\binom{t-1}{k-1} \mbox{ for }k\geq 3.
\]

Thirdly we prove \eqref{eq:pi-global-gap}. The equivalent form is
\begin{align}\label{ineq-6.5}
\binom{n-2}{k-2}-\binom{n-k-2}{k-2}-n>\binom{t-1}{k-1}.
\end{align}
Let $g(n):=\binom{n-2}{k-2}-\binom{n-k-2}{k-2}-n$. Note that
\[
g(n+1)-g(n) =\binom{n-2}{k-3}-\binom{n-k-2}{k-3}-1>0 \mbox{ for }k\geq 4.
\]

Let $n\geq 2k+5$. Then the right hand side of \eqref{ineq-6.5} is $\binom{k+1}{k-1}=\binom{k+1}{2}$, independent of $n$. As $g(n)$ is monotone increasing in $n$, it is sufficient to prove
\[
g(2k+3)=\binom{2k+3}{k-2}-\binom{k+3}{k-2}-2k-5>\binom{k+1}{2}.
\]
Noting $\binom{2k+3}{k-2}-\binom{k+3}{k-2} > \binom{2k+2}{k-3}>\binom{2k+2}{2}$ for $k\geq 5$, we have
\[
g(2k+3)>\binom{2k+2}{2} -2k-5>\binom{k+1}{2},
\]
which is true for $k\geq 2$.
For $n=2k+1$ or $2k+2$, the right hand side of \eqref{ineq-6.5} is $1$. Using the monotonicity of $g(n)$, it is sufficient to consider the $n=2k+1$ case:
\[
g(2k+1)=\binom{2k-1}{k-2}-\binom{k-1}{k-2}-2k-1>1.
\]
Equivalently,
\[
\binom{2k-1}{k-2}>3k+1.
\]
As $\binom{2k-1}{k-2}\geq \binom{2k-1}{3}$ for $k\geq 5$ and $\binom{2k-1}{3} >3k+1$ for $k\geq 4$, \eqref{eq:pi-global-gap} holds.

Similarly for $2k+3\leq n\leq 2k+4$, it is sufficient to check the $n=2k+3$ case. In equivalent form,
\[
\binom{2k+1}{k-2}-\binom{k+1}{k-2}>3k+3.
\]
Now for $k\geq 5$,
\[
\binom{2k+1}{k-2}-\binom{k+1}{k-2}\geq \binom{2k}{k-3}\geq \binom{2k}{2}=k(2k-1)>3k+3,
\]
completing the proof.
\end{proof}

\begin{proof}[Proof of Lemma \ref{lem-6.1}]
First we prove \eqref{ineq-6.1}.
By \eqref{ineq-key0} we have the following two inequalities:
\begin{align*}
\binom{n-k-5}{k-2}&>\frac{n-2-(k+3)(k-2)}{n-2} \binom{n-2}{k-2},\\[3pt] \binom{n-k-1}{k-2}&>\frac{n-2-(k-1)(k-2)}{n-2} \binom{n-2}{k-2}.
\end{align*}
Summing them yields
\[
\binom{n-k-5}{k-2}+\binom{n-k-1}{k-2} >\left(2-\frac{2(k+1)(k-2)}{n-2}\right)\binom{n-2}{k-2}>\frac{5}{4}\binom{n-2}{k-2}> \binom{n-2}{k-2}+1.
\]
Thus \eqref{ineq-6.1} holds.

Next we prove \eqref{ineq-6.2}. Since $\binom{n-k-2}{k-2}=\frac{n-k-2}{k-2}\binom{n-k-3}{k-3}$ and
\[
\frac{n-k-2}{k-2} \ge \frac{n-k^2-k+4}{k-2}=\frac{n-2}{k-2}-k-3,
\]
so it is enough to prove
\begin{align}\label{ineq-7.1}
\binom{n-k-3}{k-3}\ge \binom{k^2+3k-10}{k-1}.
\end{align}
Now
\[
\frac{\binom{n-k-3}{k-3}}{\binom{k^2+3k-10}{k-1}}
=
\frac{(k-1)(k-2)}{(k^2+2k-9)(k^2+2k-8)}
\prod_{i=0}^{k-4}\frac{n-k-3-i}{k^2+3k-10-i}.
\]
For \(0\le i\le k-4\),
\[
\frac{n-k-3-i}{k^2+3k-10-i}
\ge
\frac{n-k-3}{k^2+3k-10}
>
\frac{12k^2-k-3}{k^2+3k-10}
>9,
\]
(using \(n>12k^2\) and \(k\ge7\)). Therefore,
\[
\frac{\binom{n-k-3}{k-3}}{\binom{k^2+3k-10}{k-1}}
>
\frac{(k-1)(k-2)}{(k^2+2k-9)(k^2+2k-8)}\,9^{k-3}
>
\frac{1}{8k^2}\,9^{k-3}.
\]
For \(k\ge7\), \(\frac{1}{8k^2}\,9^{k-3}\ge \frac{9^4}{8\times 7^2}>1\).
Thus \eqref{ineq-7.1} holds and \eqref{ineq-6.2} follows.

Thirdly  we prove \eqref{ineq-6.3}. Note that
\[
\frac{\binom{n-4}{k-3}}{\binom{k^2+2k-4}{k-1}}
=
\frac{(k-1)(k-2)}{(k^2+2k-4)(k^2+2k-5)}\prod_{i=0}^{k-4}\frac{n-4-i}{k^2+2k-6-i}.
\]
For \(k\ge 7\) and \(n>12k^2\), each term in the product satisfies
\[
\frac{n-4-i}{k^2+2k-6-i}\ge
\frac{n-4}{k^2+2k-6}>10.
\]
Hence
\[
\frac{\binom{n-4}{k-3}}{\binom{k^2+2k-4}{k-1}}
>
\frac{(k-1)(k-2)}{(k^2+2k-4)(k^2+2k-5)(k^2+2k-6)}\,10^{k-4}(n-4).
\]
Using
\[
\frac{(k-1)(k-2)}{(k^2+2k-4)(k^2+2k-5)(k^2+2k-6)}
\ge\frac{1}{4k^4} \mbox{ for }k\geq 7,
\]
we get
\[
\frac{\binom{n-4}{k-3}}{\binom{k^2+2k-4}{k-1}}
>\frac{10^{k-4}}{4k^4}(n-4).
\]
For \(k\ge 8\), \( \dfrac{10^{k-4}}{4 k^4}>\dfrac1{k-1}\), and since \(n-k^2<n-4\), it follows that
\[
\binom{n-4}{k-3}
>\frac{n-k^2}{k-1}\binom{k^2+2k-4}{k-1}.
\]
Thus \eqref{ineq-6.3} holds for $n\geq 12k^2$, $k\geq 8$.
For $k=7$, \eqref{ineq-6.3} is equivalent to
\[
\binom{n-4}{4}
>\frac{n-49}{6}\binom{59}{6},
\]
which is true when  $n\geq 556$. Thus \eqref{ineq-6.3} holds.

Finally, let us prove \eqref{ineq-6.4} for $n\geq k^2$ and $k\geq 3$. Note first
\[
\frac{\binom{n-k-2}{k-2}}{\binom{n-k-3}{k-3}} =\frac{n-k-2}{k-2} \geq \frac{k^2-k-2}{k-2}=k+1.
\]
To conclude the proof we show
\[
\frac{\binom{n-4}{k-3}}{\binom{n-k-3}{k-3}}<e.
\]
Equivalently,
\begin{align}\label{ineq-7.4}
\prod_{0\leq i\leq k-4} \frac{n-4-i}{n-k-3-i}<e.
\end{align}
Noting $n-2k+1\geq (k-1)^2$, each term in the product is at most
\[
1+\frac{k-1}{n-2k+1}\leq 1+\frac{1}{k-1}.
\]
Hence the left hand side of \eqref{ineq-7.4} is at most $(1+\frac{1}{k-1})^{k-3}<e$, as desired. Thus \eqref{ineq-6.4} holds.
\end{proof}

\section{Concluding remarks}

Let us close this paper with a conjecture.

\begin{conj}
Suppose that $\mathcal{F} \subset \binom{[n]}{k}$
is non-trivial and triangle-free. Then for \(n \geq 3k-3 \geq 6\),
\[
    |\mathcal{F}| \leq |\mathcal{T}|.
\]
\end{conj}

\vspace{8pt}
{\noindent \bf Acknowledgement.} The second author was supported by the
National Natural Science Foundation of China Grant no. 12471316 and the Fundamental Research Funds for
 the Central Universities.


\begin{thebibliography}{10}

\bibitem{BD}
B. Bollob\'{a}s and P. Duchet, Helly Families of Maximal Size, J. Combin. Theory Ser. A, 26 (1979), 197--200.

\bibitem{CambieSalia}
S.~Cambie and N.~Salia, Set systems without a simplex, Helly hypergraphs and
union-efficient families, arXiv:2210.16211, 2022.

\bibitem{E}
P. Erd\H{o}s: Topics in combinatorial analysis. Proc. Second Louisiana Conf. on Comb.,
Graph Theory and Computing (R. C. Mullin et al., Eds.) pp. 2--20, Louisiana State
Univ., Baton Rouge, 1971.

\bibitem{EKR}
P. Erd\H{o}s, C. Ko, R. Rado, Intersection theorems for systems of finite sets, Quart. J. Math. Oxford Ser. 12 (1961), 313--320.

\bibitem{F76}
P. Frankl, On Sperner families satisfying an additional condition, J. Combin. Theory Ser. A 20 (1976), 1--11.

\bibitem{F81}
P. Frankl, On a problem of Chv\'{a}tal and Erd\H{o}s on hypergraphs containing no generalized
simplex, J. Combin. Theory Ser. A 30 (1981), 169--182.

\bibitem{F87}
P. Frankl, The shifting technique in extremal set theory, Surveys in Combinatorics  123 (1987), 81--110.

\bibitem{FW22}
P. Frankl, J. Wang, A product version of the Hilton-Milner-Frankl theorem, Sci. China Math., 67 (2024), 455--474.


\bibitem{HM67}
A.J.W. Hilton, E.C. Milner, Some intersection theorems for systems of finite sets, Quart.J. Math. Oxford Ser. 18 (1967), 369--384.


\bibitem{MV}
D. Mubayi and J. Verstra\"{e}te, Proof of a conjecture of Erd\H{o}s on triangles in set
systems, Combinatorica 25 (2005), 599--614.

\bibitem{Tuza94}
Zs.~Tuza, Largest size and union of Helly families,
Discrete Math. 127 (1994), 319--327.

\end{thebibliography}
\end{document}